\documentclass[12pt,a4paper,reqno]{amsart}
\usepackage[all]{xypic}
\usepackage{amsmath,amsfonts,amssymb,amsthm,xcolor,amscd,tikz,xspace,enumerate}
\usepackage[mathscr]{eucal}
\usepackage[pagebackref, colorlinks=true,linkcolor=blue,citecolor=red]{hyperref}
\usepackage{pdfpages}
\usetikzlibrary{arrows}
\usetikzlibrary{arrows.meta}
\usetikzlibrary{intersections}
\usepackage{pgfplots}
\usetikzlibrary{calc,3d,shapes, pgfplots.external, intersections}
\usepackage{pgfplots}
\usepackage{braket}

\usetikzlibrary{patterns}
\usetikzlibrary{shadings,intersections,calc}

 \newtheorem{theorem}{Theorem}[section]

 \newtheorem{corollary}[theorem]{Corollary}
 \newtheorem{lemma}[theorem]{Lemma}
 
 \newtheorem{definition}[theorem]{Definition}
\theoremstyle{remark}
\newtheorem{remark}[theorem]{Remark}
 
 \numberwithin{equation}{section}

\def\.{{\cdot}}
\def\hat{\widehat}

\newcommand\Z{{\mathbb Z}}
\def\<{\langle} \def\>{\rangle}

\begin{document}

\title[Decomposition of Pauli groups via weak central products]
 {Decomposition of Pauli groups\\ 
 via weak central products}

\author[A. Rocchetto]{Andrea Rocchetto$^\flat$}
\address{$^\flat$  Department of Computer Science\endgraf
\ \ University of Texas at Austin\endgraf
\ \ 2317 Speedway, Austin, TX 78712 United States\endgraf
\ \ Email: \texttt{andrea@cs.utexas.edu}}

%----------Author 2
\author[F.G. Russo]{Francesco G. Russo$^\sharp$}
\address{$^\sharp$ Department of Mathematics and Applied Mathematics\endgraf
\ \ University of Cape Town\endgraf
\ \ Private Bag X1, 7700, Rondebosch, Cape Town, South Africa\endgraf
\ \  Email: \texttt{francescog.russo@yahoo.com}}



\date{\today}

\begin{abstract}
For any $m \ge 1$ and odd prime power $\mathtt{q}=\mathtt{p}^m$, for $\mathtt{q}=2$, and for any $n \ge 1$, we show a result of decomposition for Pauli groups $\mathcal{P}_{n,\mathtt{q}}$ in terms of weak central products. This can be used to describe the underlying structure of Pauli groups on $n$ qudits of dimension $\mathtt{q}$ and enables us to identify  abelian subgroups of $\mathcal{P}_{n,\mathtt{q}}$. As a consequence of our main results, we show a similar factorisation for the `lifted' Pauli groups recently  introduced by Gottesman and Kuperberg in the context of error-correcting codes in quantum information theory.
\end{abstract}

%%% ----------------------------------------------------------------------
\maketitle
%%% ----------------------------------------------------------------------

\section{Introduction}

The algebraic properties of the Pauli group find several applications in quantum information and computation. Notable examples are quantum error correcting codes~\cite{ gottesman2010introduction, gottesman1997stabilizer, haah}, efficient classical simulation~\cite{aaronson}, and the theory of mutually unbiased basis~\cite{durt}. Central to these applications are the abelian subgroups of the Pauli group. 

Identifying abelian subgroups is in general a nontrivial problem. Here we present an algebraic decomposition of the Pauli group that enables us to identify its abelian subgroups. Our main observation is that nonabelian extraspecial $\mathtt{p}\,$-groups admit decompositions via central products and this makes easier the identification of their abelian subgroups. 
Although the Pauli group does not belong to this class, it is close to be nonabelian extraspecial.
In fact, it is a generalised nonabelian extraspecial $\mathtt{p}\,$-group. 
These types of groups admit a decomposition result in terms of weak central products.
Leveraging on this result we prove a decomposition for Pauli groups on prime power qudits that offers relevant information on the structure of its subgroups and quotients and highlights the relation between the Heisenberg and Pauli groups.

As a further application of our techniques we show a decomposition result for the `lifted' Pauli groups introduced by Gottesman and Kuperberg in the context of quantum error correcting codes for qudits ~\cite{gottesmankuperberg, gottesmantalk}.

The paper is structured as follows. We begin in Section~\ref{sec:prelims} by recalling some facts on the Heisenberg and Pauli groups. In Section~\ref{sec:Pauli} we prove some lemmas of computational nature and additional facts which will be used for the proofs of the main results, which we present in Section~\ref{sec:main}. We conclude in Section~\ref{sec:applications} where we show a decomposition for `lifted' Pauli groups and discuss the problem of finding abelian subgroups in Pauli groups.

\subsection*{Notation}\label{sec:prelims}

For an arbitrary group $G$, the set $Z(G)=\{a\in G \ | \ ab=ba, \ \forall b \in G\}$ denotes the \textit{center} of  $G$.
$Z(G)$ is a normal subgroup of $G$. For $a,b \in G$ let $[a,b]=a^{-1}b^{-1}ab$ be the \textit{commutator} between $a$ and $b$.
For the commutator we also use the notation $[a,b] =a^{-1}a^b$, where $a^b=b^{-1}ab$ denotes the conjugate of $a$ by $b$.
The \textit{derived subgroup} of a group $G$ is its smallest subgroup containing all $[a,b]$ and is denoted by $[G,G]=\langle [a,b] \ |  \  a,b \in G\rangle$.  As usual, a  subgroup $H$ of $G$ is \textit{maximal}, if $H \neq G$ and for any subgroup $K$ of $G$ such that $H \subseteq K \subseteq G$, then either $K=H$ or $K=G$. The \textit{Frattini subgroup} of $G$ is the intersection of all maximal subgroups of $G$ and is denoted by $\Phi(G)$. $\Phi(G)$ is normal in $G$. Finally, given two subgroups $H$ and $K$ of $G$ we say that  $G=H\rtimes K$ is the \textit{semidirect product} of  $H$ and $K$, if $H\cap K=1$, $H$ is normal in $G$, and $HK=G$. We refer the reader to~\cite{hofmor, kos, rob} for a more detailed presentation of these notions.  %\MAGENTA{Giusto per avvertire il lettore che ci sono manuali dove si possono trovare queste nozioni}

\section{Preliminaries}

We begin with a brief review of relevant notions on the Heisenberg and Pauli groups.

\subsection{The Heisenberg group}
Consider  a field $\mathbb{F}$  of  characteristic $\neq 2$ and triples $(p,q,t) \in \mathbb{F}^3=V$ in the finite dimensional vector space $V$ over $\mathbb{F}$  with  a nondegenerate skew symmetric  bilinear form \( \alpha : ((p_1,q_1),(p_2,q_2))  \in \mathbb{F}^2 \times  \mathbb{F}^2 \to \alpha ((p_1,q_1),(p_2,q_2)) \in \mathbb{F}\). The Heisenberg group $\mathbb{H}(V,\alpha)$ (or simply $\mathbb{H}(\mathbb{F})$) is the set $\mathbb{F}^3$ endowed with the binary operation
\begin{equation} \label{h1}
\square  \ :  \ ((p_1, q_1, t_1), (p_2, q_2, t_2)) \in \mathbb{F}^3 \times \mathbb{F}^3  \mapsto (p_1, q_1, t_1) \ \square \ (p_2, q_2, t_2) \end{equation}
$$ =\ \Big(p_1 + p_2, \ \  q_1 + q_2,  \ \ t_1 + t_2
+ \alpha((p_1,q_1),(p_2,q_2))
  \Big) \in \mathbb{F}^3.$$
See terminology and notations in \cite[Definition 2.1, 2.6 and Theorems 2.4, 2.5]{grundh}. It is common to choose \[\alpha((p_1,q_1),(p_2,q_2)))=p_1q_2-q_1p_2.\]
  Since \(\alpha((p_1,q_1),(p_2,q_2))\) depends on \((p_1,q_1)\) and \((p_2,q_2)\) and is skew symmetric bilinear, $\square$ is not a commutative operation. Moreover $\alpha((p_1,q_1),(p_2,q_2)))$ is alternating, that is, $\alpha((p_1,q_1),(p_1,q_1)))=\alpha((p_2,q_2),(p_2,q_2)))=0$.
One can see that $(\mathbb{F}^3, \square)$ is a group and in fact it is the  \textit{Heisenberg group} $\mathbb{H}(\mathbb{F})$, which admits the more convenient  matrix representation 
\begin{equation} \label{h2}
\mathbb{H}(\mathbb{F}) = \left\{ \left(\begin{array}{ccc}
1 & p & t\\
0 & 1 & q\\
0 & 0 & 1\\
\end{array}
\right) \ | \ p,q,t \in \mathbb{F} \right\}= \{M(p,q;t) \ | \ p,q,t \in \mathbb{F}\}
\end{equation}
with respect to the usual matrix product in the general linear group $\mathrm{GL}_3(\mathbb{F})$. A way to check that $(\mathbb{F}^3, \square)$ is isomorphic to \(\mathbb{H}(\mathbb{F})\) (with respect to the usual product of matrices) is given in \cite[Theorem 6.1]{grundh}. We sketch the main ideas here for convenience of the reader. The first step is to consider the following map:
\[\varphi : (p, q, t) \in \mathbb{F}^3 \mapsto \varphi(p,q,t)= M(p,q;\frac{1}{2}(t+pq)) \in \mathbb{H}(\mathbb{F})\]
 which satisfies the conditions
\[\varphi ((p_1, q_1, t_1) \, \square \, (p_2, q_2, t_2)) =\varphi((p_1+p_2, q_1+q_2, t_1 + t_2 + (p_1q_2-q_1p_2))\]
\[=M\left( p_1+p_2, q_1 +q_2 ; \frac{1}{2} (t_1+t_2 + (p_1q_2-q_1p_2)+ (p_1+p_2)(q_1+q_2)) \right) \]
\[=M \left( p_1+p_2, q_1 +q_2 ; \frac{1}{2}(t_1+t_2 +p_1q_1+p_2q_2) +p_1q_2 \right)\]
\[=M \left( p_1, q_1 ; \frac{1}{2}(t_1 + p_1q_1) \right) \ M \left( p_2, q_2 ; \frac{1}{2}(t_2 + p_2q_2) \right)= \varphi (p_1, q_1, t_1) \ \varphi (p_2, q_2, t_2) \]
and so it is a group homomorphism. The second step is to check   that $\varphi$ is  bijective and this is  routine. In particular, one can see that  the center of $\mathbb{H}(\mathbb{F})$  is the nontrivial proper normal subgroup    
\begin{equation}\label{h3}
Z(\mathbb{H}(\mathbb{F}))=\{M(0,0;t) \ | \ t \in \mathbb{F}\}
\end{equation}
so the group is nonabelian.
It is easy to see that
\begin{equation}\label{h4}
[[\mathbb{H}(\mathbb{F}),\mathbb{H}(\mathbb{F})],\mathbb{H}(\mathbb{F})]=1,  \ \ [\mathbb{H}(\mathbb{F}),\mathbb{H}(\mathbb{F})]=Z(\mathbb{H}(\mathbb{F}))\simeq \mathbb{F}, \ \ \mathbb{H}(\mathbb{F})/Z(\mathbb{H}(\mathbb{F})) \simeq \mathbb{F}^2.
\end{equation}
The condition $[[\mathbb{H}(\mathbb{F}),\mathbb{H}(\mathbb{F})],\mathbb{H}(\mathbb{F})]=1$ means that $\mathbb{H}(\mathbb{F})$ is \textit{nilpotent of class two}, following a classical terminology in group theory (see \cite{rob}).
The other two conditions in~\eqref{h4} show that the derived subgroup coincides with the center and has the size of the ground field $\mathbb{F}$, but the whole group factorized through the center has a size which is  double $\mathbb{F}^2=\mathbb{F} \oplus \mathbb{F}$ than the ground field.

We can deduce these structural properties of $\mathbb{H}(\mathbb{F})$ by the smallest numbers of generators and relations, that is, via  group presentations for $\mathbb{H}(\mathbb{F})$ (see \cite{grundh,hofmor, rob}). For instance, if $\mathbb{F}=\mathbb{F}_\mathtt{p}=\mathbb{Z}(\mathtt{p})$ is the finite field of order $\mathtt{p}$ (with odd $\mathtt{p}$), then
\begin{equation}\label{h5}
\mathbb{H}(\mathbb{Z}(\mathtt{p})) = \langle M(0,1;0), M(1,0;0), M(0,0;1) \ | \ [M(0,1;0), M(1,0;0)]=M(0,0;1) \rangle 
\end{equation}
This group has order $\mathtt{p}^3$. See \cite{rob, gap} for  basic notions on presentation on groups. Note that   \eqref{h5} is realised with the minimal number of generators and relations for $\mathbb{H}(\mathbb{Z}(\mathtt{p}))$. In fact  from \eqref{h4} one has $[M(0,1;0), M(0,0;1)]=[M(1,0;0),M(0,0;1)]=I$.  On the other  hand, it might be useful to add the redundant relations
\(M(0,0;1)^\mathtt{p}= M(1,0;0)^\mathtt{p}=M(0,1;0)^\mathtt{p}=I\)
to \eqref{h5}, since they help to understand that the group is generated by elements of prime order. Of course, one can write $\mathbb{H}(\mathbb{F})$ when $\mathbb{F}$ is a general field (eventually of characteristic zero). In this case \eqref{h5} works perfectly but  \(M(0,0;1)^\mathtt{p}= M(1,0;0)^\mathtt{p}=M(0,1;0)^\mathtt{p}=I\)
are no longer valid. More details can be found in \cite[Theorem 2.5 and Lemma 5.5]{grundh}.

Having a presentation for $\mathbb{H}(\mathbb{F})$ and knowing the meaning of generators and relations helps to understand the behaviour of a dynamical system whose group of symmetries is described by $\mathbb{H}(\mathbb{F})$. However it is possible to get similar information via the notion of semidirect product. If we look at  $\mathbb{H}(\mathbb{F})$, one can show that for any choice of  $\mathbb{F}$  of characteristic $\neq 2$, there are only two abelian maximal subgroups 
$$A = \langle M(0,0;1)\rangle \oplus \langle M(1,0;0) \rangle \simeq \mathbb{F}^2 \ \ \mbox{and} \ \ B = \langle M(0,0;1)\rangle \oplus \langle M(0,1;0) \rangle \simeq \mathbb{F}^2$$
such that $A \cap B = Z(\mathbb{H}(\mathbb{F}))$, $A \cap \langle M(0,1;0) \rangle=1$, $B \cap \langle M(1,0;0) \rangle=1$, and
 \begin{equation}\label{h6}
\mathbb{H}(\mathbb{F})= A \rtimes \langle M(0,1;0) \rangle =B \rtimes \langle M(1,0;0) \rangle \simeq  \mathbb{F}^2 \rtimes \mathbb{F} .
\end{equation}
More details can be found in \cite{grundh, hofmor}. In particular,   the following diagram is well known and describes the placement of the aforementioned subgroups in the lattice of subgroups of $\mathbb{H}(\mathbb{F})$.
\begin{center}
\begin{tikzpicture}
%[scale=.8,auto=left,every node/.style={circle,fill=blue!10}]
[scale=.8,auto=left]
   \node (d) at (4,3) {$\mathbb{H}(\mathbb{F})$};
  \node (b) at (-0.5,2) {$A$};
  \node (e) at (-2,-0.5) {$\langle M(1,0;0) \rangle $};
  \node (a) at (8,2) {$B$};
  \node (f) at (10,-0.5) {$\langle M(0,1;0) \rangle $};
  \node (c) at (4,-0.5) {$Z(\mathbb{H}(\mathbb{F}))=[\mathbb{H}(\mathbb{F}),\mathbb{H}(\mathbb{F})]$};
%  \node (g) at (4,-5) {$[\mathbb{H}(\mathbb{F}),\mathbb{H}(\mathbb{F})]$};
  \node (h) at (4,-4) {$1$};

%  \node (Fig) at (-10,2) {$\Gamma(S_3)$};
  
  %\draw (a) -- (b) -- (c)--(a);
  \draw (d) -- (a);
  \draw (d) -- (b);
%  \draw (d) -- (c);
  \draw (e) -- (b);
  \draw (f) -- (a);
%  \draw (g) -- (c);
  \draw (h) -- (c);
  \draw (e) -- (h) ;
  \draw (f) -- (h) ;
\draw (a) -- (c) ;
\draw (b) -- (c) ;

 \end{tikzpicture}
\centerline{\textbf{Figure 1}:  Portion of an Hasse diagram of $\mathbb{H}(\mathbb{F})$. }
 \end{center}
%\end{center}
\vspace{0.5em}
Note that all the subgroups involved in Fig.~1 are abelian, except $\mathbb{H}(\mathbb{F})$. At the first level we find the trivial subgroup. At the second level there are three subgroups isomorphic to the additive group of the ground field $\mathbb{F}$. At the third level there are just two subgroups isomorphic to the additive group $\mathbb{F}^2$. At the fourth level we find the whole group. 
Finally, note that the Hasse diagram in Fig.~1 presents only the subgroups that can be directly deduced from \eqref{h6} and not all the subgroups of $\mathbb{H}(\mathbb{F})$.

\subsection{The Pauli group}

Consider a finite field $\mathbb{F_{\mathtt{q}}}$ of prime power order $\mathtt{q}=\mathtt{p}^m$ of characteristic $\mathtt{p} \neq 2$.
Define the \textit{shift} operator as
\begin{equation}
    \forall j \in \mathbb{F_{\mathtt{q}}} \quad X \ket{j} = \ket{j+1}
\end{equation}
and the \textit{clock} operators as
\begin{equation}
    \forall j \in \mathbb{F_{\mathtt{q}}} \quad Z \ket{j} = \omega^j  \ket{j}
\end{equation}
where $\omega = \exp(2 \pi i /\mathtt{q})$.
$X$ and $Z$ generate a group which is known as the \textit{Pauli group}.
The Pauli group can also be defined when $\mathtt{q}=2$. In this case, we must include $i$ as a generator.

The Pauli group for $n$ qudits of odd dimension $\mathtt{q} = \mathtt{p}^m$  with $\mathtt{p}$ odd prime is 
\begin{equation}\label{new1}
\mathcal{P}_{n,\mathtt{q}}=\{\omega^c \bigotimes^n_{k=1} X^{\alpha_k} Z^{\beta_k} \ | \ c \in \mathbb{F}_\mathtt{p}, \alpha_k, \beta_k \in \mathbb{F}_\mathtt{q}\}.
\end{equation}
For $\mathtt{q}=2$ the above expression becomes
\begin{equation}\label{new2}
\mathcal{P}_{n,2}=\{i^d \bigotimes^n_{k=1} X^{\alpha_k} Z^{\beta_k} \ | \ d \in \mathbb{F}_4, \alpha_k, \beta_k \in \mathbb{F}_2\},
\end{equation}
%and the result is for $n=1$ the well known $Pauli$ $group$ of order $4^2$ on $1$ $qubit$
%\begin{equation}\label{h10}
%\mathcal{P}_{1,2}=\langle {(-I)}^{\lambda} i^{pq} w(p, q): p,q,\lambda\in \Z(2)\rangle,
%\end{equation}
%where the second index in the subscript of $\mathcal{P}_{1,2}$ refers to the fact that a qubit is a two level system (for $\mathtt{q}$-dimensional qudits we have $\mathcal{P}_{1,\mathtt{q}}$).
which is the well known Pauli group of order $4^{n+1}$ on $n$ qubits.
%\begin{equation}\label{h11}
%\mathcal{P}_{n,2}=\langle  {(-I)}^{\lambda} w(p_1, q_1) \otimes \cdots
%\otimes w(p_n, q_n) \  | \ p_1,q_1, \ldots, p_n,q_n,\lambda \in \Z(2)\rangle.
%\end{equation}
Elements of the Pauli group are idempotent so that for a $P \in \mathcal{P}_{n,\mathtt{q}}$ we have $P^{\, \mathtt{q}} = I$.

We observe that \eqref{new1} and \eqref{new2}, as well as \eqref{h5}, do not give any structural information on the groups, despite the fact that they describe the groups in terms of generators and elements or the interactions among group elements.
For this reason, it is diffcult to identify abelian subgroups and abelian quotients of the Pauli group.
As we discuss in the next sections, both the Pauli groups and the Heisenberg groups have a similar underlying structure that can be brought to light using a central product description.

In the next sections we we introduce a purely algebraic characterisation of Pauli groups.
Geometric aspects were studied in~\cite{GHW, jp, Woo1, Woo2, z} using tools from projective geometries.

\section{The formalism of Pauli groups in computational group theory}
\label{sec:Pauli}

We present some arguments of finite group theory, adapted to the present context, in order to describe $\mathcal{P}_{n,2}$. First of all we recall that the quotient group $G/[G,G]$ of an arbitrary group $G$ is always  abelian, but not necessarily of the form  $$G/[G,G]  \simeq \underbrace{\mathbb{Z}(\mathtt{p}) \oplus \mathbb{Z}(\mathtt{p}) \oplus \ldots \oplus \mathbb{Z}(\mathtt{p})}_{r-\mbox{times}},$$
that is, $\mathtt{p}$-elementary abelian of rank $r$. The presence of an elementary abelian quotient occurs for $\mathcal{P}_{1,2}$ in the sense of the following result.

 \begin{lemma}\label{l:1} The Pauli group $\mathcal{P}_{1,2}$ can be  presented  both
by \[\langle X,Y,Z \ | \ X^2=Y^2=Z^2=1, (YZ)^4=(ZX)^4=(XY)^4=1 \rangle \]
and by 
\[\mathcal{P}_{1,2} = \langle u,a,b \mid u^4 = a^2 = 1, u^2 = b^2, a^{-1}ua = u^{-1}, ub = bu, ab = ba \rangle, 
\]
where
\[u=XY, \  \ a=Y,  \ \   \mbox{and} \  \ b=XYZ.\]
Moreover, $\mathcal{P}_{1,2}$ has no elements of order $8$,  $Z(\mathcal{P}_{1,2}) \simeq \mathbb{Z}(4)$, $ [\mathcal{P}_{1,2},\mathcal{P}_{1,2}] \subseteq Z(\mathcal{P}_{1,2})$, $[\mathcal{P}_{1,2}, \mathcal{P}_{1,2}]=\Phi(\mathcal{P}_{1,2})$ and $\mathcal{P}_{1,2}/Z(\mathcal{P}_{1,2})$ is $2$-elementary abelian of rank $2$. 
\end{lemma}

\begin{proof}
Let us consider the Pauli matrices
\begin{equation}\label{paulimatricesxyz} X=\left(
\begin{array}{rr}
0 & 1 \\
1 & 0%
\end{array}%
\right) ,Y=\left(
\begin{array}{rr}
0 & -i \\
i & 0%
\end{array}%
\right) \text{ and }Z=\left(
\begin{array}{rr}
1 & 0 \\
0 & -1%
\end{array}%
\right)
\end{equation}
and check that $
X^2=Y^{2}=Z^{2}=I=\left(
\begin{array}{cc}
1 & 0 \\
0 & 1%
\end{array}%
\right) $ so the first relation is satisfied. Now one can check  that also the other relations $(YZ)^4=(ZX)^4=(XY)^4=1$ are satisfied, via the usual product of  matrices. In addition $X^2=1$ implies $X=X^{-1}$ and similarly $Y=Y^{-1}$ and $Z=Z^{-1}$  so one can find the  following calculus rules
$${(XYZ)}^4=1, \ \, [XYZ,X]=1, \ \, [XYZ,Y]=1, \ \, [XYZ,Z]=1,$$
which show that $Z(\mathcal{P}_{1,2})=\langle XYZ \rangle \simeq \mathbb{Z}(4)$. The absence of elements of order 8  can be checked either directly or looking at the relations in $\mathcal{P}_{1,2}$, this means that the maximum order of an element in $\mathcal{P}_{1,2}$ cannot exceed 4, or, equivalently, that the exponent of $\mathcal{P}_{1,2}$ is $4$.  In order to do this, one can introduce
\[
u=XY, \  \ a=Y,  \ \   \mbox{and} \  \ b=XYZ,
\]
 and show that the original presentation is equivalent to the following 
\begin{equation}\label{g1}
\mathcal{P}_{1,2} = \langle u,a,b \mid u^4 = a^2 = 1, u^2 = b^2, a^{-1}ua = u^{-1}, ub = bu, ab = ba \rangle. 
\end{equation}
Here, 
\begin{equation}\label{d8}
D_8=\langle u,a \rangle= \langle u,a \ | \ u^4=a^2=1, a^{-1}ua = u^{-1}\rangle  \ \ \mbox{and} \ \ \mathbb{Z}(4)= \langle b  \ | \ b^4=1\rangle  
\end{equation}  describe the dihedral group of order eight and the cyclic of order four, respectively. By motivations of order we get $|\mathcal{P}_{1,2}/Z(\mathcal{P}_{1,2})|=4$, concluding that  $[\mathcal{P}_{1,2},\mathcal{P}_{1,2}] \subset Z(\mathcal{P}_{1,2})$ and that no elements of order 8 are contained in $\mathcal{P}_{1,2}$.

 Further computations show that $\mathcal{P}_{1,2}/Z(\mathcal{P}_{1,2}) \simeq \mathbb{Z}(2) \times \mathbb{Z}(2)$. In addition, one can check directly that $\Phi(\mathcal{P}_{1,2})=[\mathcal{P}_{1,2},\mathcal{P}_{1,2}]$. 
\end{proof}

The previous results offers a fast and efficient way to detect abelian subgroups in $\mathcal{P}_{1,2}$. Due to the decomposition that we have described, we can look at a portion of its Hasse diagram recognising $D_8=\langle a, u\rangle$ and $\mathbb{Z}(4)=\langle b \rangle$ and drawing the Hasse diagram of $D_8$ along with the information we have in the proof of Lemma \eqref{l:1}.

\begin{center}
\begin{tikzpicture}
%[scale=.8,auto=left,every node/.style={circle,fill=blue!10}]
[scale=.8,auto=left]
   \node (d) at (0,0) {$\langle u, a,b \rangle$};
   \node (b) at (-2,-2) {$\langle u, a \rangle$};
   \node (e) at (-4,-4) {$\langle u^2, ua \rangle $};
  \node (a) at (2,-2) {$\langle u, b \rangle$};
  \node (f) at (4,-4) {$\langle b \rangle $};
\node (l) at (-4,-6) {$\langle  ua \rangle$}; 
 \node (m) at (-2,-6) {$\langle  u^3a \rangle$};
  \node (i) at (-2,-4) {$\langle u  \rangle$};
  \node (c) at (2,-4) {$\langle u^2, a \rangle$};
  \node (g) at (0,-6) {$\langle u^2 \rangle$};
  \node (h) at (0,-8) {$1$};
  \node (n) at (4,-6) {$\langle u^2a \rangle$};
    \node (o) at (2,-6) {$\langle a \rangle$};
%    \node (p) at (7,-0.3) {$\langle u^2, b \rangle$};

%  \node (Fig) at (-10,2) {$\Gamma(S_3)$};
  
  %\draw (a) -- (b) -- (c)--(a);
\draw (n) -- (h);  
\draw (o) -- (h);  
\draw (n) -- (c);  
\draw (o) -- (c);  
\draw (i) -- (g);  
  \draw (d) -- (a);
  \draw (d) -- (b);
  \draw (d) -- (c);
  \draw (e) -- (b);
  \draw (f) -- (a);
  \draw (g) -- (c);
  \draw (h) -- (g);
  \draw (e) -- (g) ;
  \draw (f) -- (g) ;
\draw (a) -- (i) ;
\draw (b) -- (c) ;
\draw (b) -- (i) ;
\draw (m) -- (h) ;
\draw (l) -- (h) ;
\draw (l) -- (e) ;
\draw (m) -- (e) ;

 \end{tikzpicture}
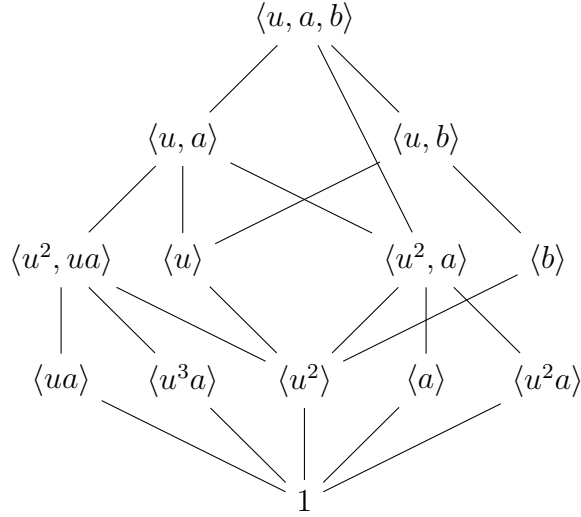
\centerline{\textbf{Figure 2}:  Portion of an Hasse Diagram of $\mathcal{P}_{1,2}$. }
 \end{center}
\vspace{0.5em}
At the first level of Fig.~2, there is the trivial subgroup; at the second level some subgroups of order two, so abelian. Then some subgroups of order $4$ again abelian at the third level. The level just below the whole group has two nonabelian subgroups.

We now investigate what happens to \eqref{new1} in case $\mathtt{p}$ is odd. A first important example can be illustrated, in order to understand the structure of Pauli groups on qudits.

\begin{remark}\label{newexample} Consider $\mathtt{p}=3$ in \eqref{new1}. Then $\mathcal{P}_{1,3}$ is a nonabelian $3$-group of order $27$.
More explicitly, the elements of $\mathcal{P}_{1,3}$ can be visualized via the following ($3\times 3$) nonsingular matrices, endowed of the usual operation of row by column, and obtained by the set multiplication 

\begin{gather*}
\{ 1,\omega, \omega^2 \} \times  \\
\Bigg\{ 
\left(\begin{array}{lll}{1} & {0} & {0} \\ {0} & {1} & {0} \\ {0} & {0} & {1}\end{array}\right), 
\; 
\left(\begin{array}{ccc}{1} & {0} & {0} \\ {0} & {e^{\frac{2 \pi i}{3}}} & {0} \\ {0} & {0} & {e^{\frac{4 \pi i}{3}}}\end{array}\right), 
\;
\left(\begin{array}{ccc}{1} & {0} & {0} \\ {0} & {e^{\frac{4 \pi i}{3}}} & {0} \\ {0} & {0} & {e^{\frac{8 \pi i}{3}}}\end{array}\right), 
\;
\left(\begin{array}{ccc}{0} & {1} & {0} \\ {0} & {0} & {1} \\ {1} & {0} & {0}\end{array}\right),
\;
\left(\begin{array}{ccc}{0} & {0} & {1} \\ {1} & {0} & {0} \\ {0} & {1} & {0}\end{array}\right),
\\
\left(\begin{array}{ccc}{0} & {e^{\frac{2 \pi i}{3}}} & {0} \\ {0} & {0} & {e^{\frac{4 \pi i}{3}}} \\ {1} & {0} & {0}\end{array}\right), 
\;
\left(\begin{array}{ccc}{0} & {e^{\frac{4 \pi i}{3}}} & {0} \\ {0} & {0} & {e^{\frac{8 \pi i}{3}}} \\ {1} & {0} & {0}\end{array}\right), 
\;
\left(\begin{array}{ccc}{0} & {0} & {e^{\frac{4 \pi i}{3}}} \\ {1} & {0} & {0} \\ {0} & {e^{\frac{2 \pi i}{3}}} & {0}\end{array}\right), 
\;
\left(\begin{array}{ccc}{0} & {0} & {e^{\frac{8 \pi i}{3}}} \\ {1} & {0} & {0} \\ {0} & {e^{\frac{4 \pi i}{3}}} & {0}\end{array}\right)
\Bigg\} .
\end{gather*}

Now a direct computation (or  looking at \cite{gap}) shows that there is only one nonabelian group of order $27$ and exponent three. See Remark 3.9 below for more details.
\end{remark}

The proof of Lemma \ref{l:1} shows that $\mathcal{P}_{1,2}$ can be decomposed into the non-semidirect product of a subgroup, which is isomorphic to $D_8$, and of another which is cyclic of order four. Product decompositions have long been studied in group theory since they allow us to control the structural properties of groups. In general, semidirect products of finite groups of the form $G=AB$ imply $|G|=|A| \cdot |B|$, but for the non-semidirect products that appear in Lemma \ref{l:1} we have that $|G|\le |A| \cdot |B|$. Some well known decompositions appear for various families of groups. For instance, abelian finitely generated groups may be decomposed in direct products of cyclic groups \cite[ 4.2.10]{rob}.  
Remak's decomposition  \cite[3.3.12]{rob}  describes further conditions of splitting in direct products and the Schur--Zassenhaus Theorem introduces a decomposition more general than those obtained with direct products \cite[9.1.2]{rob}. 

Here we give a decomposition based on the notion of weak central product. The central product is an important concept in group theory since the structure of most nilpotent groups can be described via central products (recall that the Heisenberg group is nilpotent). 
These kind of products have only recently been used to describe properties of physical systems~\cite{bagrus1}.

The weak central product is defined as
\begin{definition}[Weak central product]
\label{def:1}
A group $G$ is the weak central product of its normal subgroups $H$ and $K$, if simultaneously
\begin{itemize}
\item[(i)] $G=HK$;
\item[(ii)]$[H, K] \subseteq Z(G)$.
\end{itemize}
\end{definition}
When  Definition \ref{def:1} is satisfied with $[H, K]=C\subseteq Z(G)$, we write  $$G= H   \bullet_C  K$$ in order to specify the portion of the center which allows us to make the central product.
Note that $[H,K] \subseteq H \cap K$ above, because $H$ and $K$ are both normal in $G$, but we do not know in general whether $[H,K]=1$ or not. When   $[H,K]=H \cap K =Z(G)$, we follow \cite[Pages 145--146]{rob} calling $G$  \textit{central product} of $H$ and $K$. In this case, we use the notation $G= H   \bullet  K$, where we dropped the subscript for $C$.

\begin{remark}\label{additional}
The Heisenberg group $\mathbb{H}(\Z(\mathtt{p}))$ can be decomposed in semidirect product as in \eqref{h6}, but also as central product of $A=\langle M(1,0;0), M(0,0;1)\rangle$ and $B=\langle M(0,1;0), M(0,0;1)\rangle$, noting that (i) and (ii) are satisfied in Definition \ref{def:1}. In fact  $$[A,B] = A \cap B= \langle M(0,0;1) \rangle=Z(\mathbb{H}(\Z(\mathtt{p})))$$ in this specific case.
\end{remark}

We observe that it is possible to have (ii) of Definition \ref{def:1} as a strict inclusion. 
%so it would be good to introduce an appropriate notation.

\begin{remark}\label{additionalbis}
 From Remark \ref{additional}, note that $A \simeq B \simeq \Z(\mathtt{p}) \oplus \Z(\mathtt{p})$ and so
$$\mathbb{H}(\Z(\mathtt{p})) = A \bullet B \simeq (\Z(\mathtt{p}) \oplus \Z(\mathtt{p})) \bullet (\Z(\mathtt{p}) \oplus \Z(\mathtt{p}))$$
is an alternative description for the Heisenberg group $\mathbb{H}(\Z(\mathtt{p}))$.
\end{remark}

We recall the following definition from~\cite[Page 140]{rob}:
\begin{definition}[Extraspecial $\mathtt{p}$-group]
 A finite group $G$ of  $|G|=\mathtt{p}^n$ ($\mathtt{p}$ prime and $n \ge 1$) is \textit{extraspecial} if $Z(G)=[G,G]$ and $Z(G)$ has order $\mathtt{p}$.
\end{definition}

The structure of these groups can be described in terms of central products (Definition \ref{def:1}) and some relevant classes of $\mathtt{p}$-groups. In order to accomplish this task we recall some further results of finite group theory.

\begin{lemma}[See \cite{rob}]\label{e1e2}
If $\mathtt{p}$ is odd, any nonabelian $\mathtt{p}$-group of order $\mathtt{p}^3$ must be isomorphic either to
$$E_1= \langle x,y  \ | \ x^\mathtt{p}=y^\mathtt{p}=1, x^{-1}[x, y]x = y^{-1}[x, y]y = [x, y] \rangle, $$
which is called nonabelian $\mathtt{p}$-groups of order $\mathtt{p}^3$ and exponent $\mathtt{p}$, or to
$$E_2= \langle x,y  \ | \ x^{\mathtt{p}^2}=y^p=1, y^{-1}xy = x^{1+\mathtt{p}} \rangle, $$
which is called nonabelian $\mathtt{p}$-groups of order $\mathtt{p}^3$ and exponent $\mathtt{p}^2$.
\end{lemma}

In fact one can see that $E_1$ has no elements of order $\mathtt{p}^2$, while $E_2$ has it. Note that $\mathcal{P}_{1,3}$ in Remark \ref{newexample} is exactly $E_2$ when $\mathtt{p}=3$. In case $\mathtt{p}=2$, the situation is completely different.

\begin{lemma}[See \cite{rob}]\label{d8q8}
Any nonabelian $2$-group of order $8$ either is isomorphic to the dihedral group 
$$D_8= \langle x,y \  | \ x^4=y^2=1, y^{-1}xy=x^{-1}\rangle$$ 
 or to the quaternion group
$$Q_8= \langle i,j,k  \ | \ i^2=j^2=k^2=-1, ij=k, jk=i, ki=j \rangle.$$ 
\end{lemma}

Note that any nonabelian $\mathtt{p}$-group of order $\mathtt{p}^3$ is extraspecial, both if $\mathtt{p}$ is even and if $\mathtt{p}$ is odd. For sure, the importance of these extraspecial groups appears the moment we observe the following fact

\begin{remark}\label{nice}
We have  $\mathbb{H}(\Z(\mathtt{p})) \simeq E_2$ for $p$ odd and  $\mathbb{H}(\Z(\mathtt{2})) \simeq D_8$ for $\mathtt{p}=2$ and in both cases we have extraspecial groups. In a similar vein one can see that   $$\mathbb{H}(\Z(\mathtt{p}^n)) \simeq (\Z(\mathtt{p}^n) \oplus \Z(\mathtt{p}^n)) \rtimes \Z(\mathtt{p}^n)$$  is always  extraspecial  (eventually with $\mathtt{p}=2$) of order $\mathtt{p}^{3n}$ and center of order $\mathtt{p}$, but it is more interesting to describe 
$$\mathbb{H}({\Z(\mathtt{p})}^n) \simeq \underbrace{(\Z(\mathtt{p}) \oplus \ldots \oplus \Z(\mathtt{p}))}_{n-\mathrm{times}} \rtimes \Z(\mathtt{p})={\Z(\mathtt{p})}^n \rtimes \Z(\mathtt{p}),$$
which is again an extraspecial $\mathtt{p}$-group of order
$\mathtt{p}^{2n+1}$ and center always of order $\mathtt{p}$. Note that the group in Remark \ref{newexample} is isomorphic to $\mathbb{H}(\Z(3))$, and, more generally the Pauli group on one qudits $\mathcal{P}_{1,p}$ for odd $\mathtt{p}$ is a nonabelian $\mathtt{p}$-group of order $\mathtt{p}^3$ and exponent $\mathtt{p}$, that is, $\mathcal{P}_{1,p} \simeq E_2 \simeq \mathbb{H}(\Z(\mathtt{p}))$.
\end{remark}

Now we report the main classification of extraspecial $\mathtt{p}$-groups.

\begin{lemma}[See \cite{rob}, Exercises 6 and 7]\label{l:2} A nonabelian extraspecial group $G$ of $|G|=\mathtt{p}^{2n+1}$ has $Z(G)$ of order $\mathtt{p}$ and  $\mathtt{p}$-elementary abelian quotient $G/Z(G)$. Moreover, if $\mathtt{p}=2$, then  $G$ is the central product of $D_8$'s or a central product of $D_8$'s and a single $Q_8$. If $\mathtt{p}>2$, then  either $G$ has exponent $\mathtt{p}$, or else it is a central product $E_1$'s and a single $E_2$.
\end{lemma}

The Pauli group  $\mathcal{P}_{1,2}$ cannot be described by Lemma \ref{l:2}, in fact $|\mathcal{P}_{1,2}|=16$ and  both $\mathcal{P}_{1,2} \neq D_8 \bullet_{\mathbb{Z}(2)} D_8$ and $\mathcal{P}_{1,2} \neq D_8 \bullet_{\mathbb{Z}(2)} Q_8$. This means that $\mathcal{P}_{1,2}$ is not an extraspecial $2$-group. An alternative argument can be used if we note that  $|Z(\mathcal{P}_{1,2})|> 2$ by Lemma \ref{l:1}.

\begin{remark}\label{r:1} If $\mathtt{p}>2$, then Lemma \ref{l:2} says that $G$ may be decomposed in the central product of finitely many factors isomorphic either to $E_1$ or $E_2$. If $\mathtt{p}=2$, the same is true but now the factors must be isomorphic either to $D_8$ or to $Q_8$. Because of these restrictive conditions, it is usual to talk about the extraspecial $\mathtt{p}$-group $E_{\mathtt{p}^{2n+1}}$ of order $\mathtt{p}^{2n+1}$, up to specify if $\mathtt{p}$   is even or odd.
\end{remark}

In order to get a decomposition results for Pauli groups we need to introduce a more general notion (see \cite{rob}, Exercise 8).

\begin{definition}[Generalised extraspecial $\mathtt{p}$-group]\label{def:2}
 A finite group $G$ of order $|G|=\mathtt{p}^n$ ($\mathtt{p}$ prime and $ n \ge 1$) is generalized extraspecial (or generalized Heisenberg), if $[G,G]$ is  of order $\mathtt{p}$ and $Z(G)$ is cyclic.
\end{definition}

One can find  the folllowing description for these groups.

\begin{lemma}\label{l:3} A  generalized extraspecial $\mathtt{p}$-group $G$ satisfies the  conditions:
\begin{itemize}
\item[(i)] $[G,G] \subseteq Z(G)$ and $G/Z(G)$ is elementary abelian of even rank; 
\item[(ii)] $G$ can be decomposed in weak central product;
\item[(iii)] $G=H/L$, where $H= E \times C$ with $E$ extraspecial, $C$ cyclic and $L$ is a subgroup of $H$ of exponent $\mathtt{p}$;
\item[(iv)] $G$ is nonabelian but all of its proper quotients are abelian (i.e.: $G$ is just nonabelian). 
\end{itemize}
\end{lemma}

\begin{proof} 
(i) and (ii) follow from the definitions and are mentioned in \cite[Exercises 8 and 9]{rob}. (iii) and (iv) can be deduced by \cite[Theorem 11.2]{kos}.  
\end{proof}

Note that Lemma \ref{l:3} does not guarantee that a nonabelian generalized extraspecial group has  center of prime order. In fact, the moment this is true, we get the case of extraspecial groups since the condition $[G,G]=Z(G)$ is automatically satisfied. Therefore Definition \ref{def:2} is more general than Definition \ref{def:1}.

\begin{remark}\label{r:2}
The Pauli group $\mathcal{P}_{1,2}$ is the central product of
 $D_8$ by $\mathbb{Z}(4)$ as noted in Lemma \ref{l:1}, that is, $\mathcal{P}_{1,2}= D_8 \bullet_{\mathbb{Z}(2)} \mathbb{Z}(4)$.
In particular, $Z(\mathcal{P}_{1,2}) \simeq \mathbb{Z}(4)$  and so $\mathcal{P}_{1,2}$ is not extraspecial by Lemma \ref{l:2}. On the other hand, $\mathcal{P}_{1,2}$ satisfies the conditions in Definition \ref{def:2} so it is generalized extraspecial. Note that here $D_8 \cap \mathbb{Z}(4)= \mathbb{Z}(2)$ is properly contained in $Z(\mathcal{P}_{1,2})$ and  Definition \ref{def:1} is indeed satisfied.
\end{remark}

\section{Main results}
\label{sec:main}

The first main result is a direct consequence of what we noted in Lemmas \ref{l:1} and  \ref{l:3}. Recall that \textit{just nonabelian groups} are groups that are nonabelian but whose all proper quotients are abelian (they have been extensively studied in~\cite{kos, russo}). A dual case is when we have a group which is nonabelian but all of its proper subgroups are abelian. These are called \textit{minimal nonabelian groups} and were classified by Redei and Schmidt (see \cite{rob}).
Both definitions are important in the context of Pauli groups.

\begin{theorem}\label{main1additional} The Pauli group $\mathcal{P}_{1,2}$ is just nonabelian but is not minimal nonabelian. The Pauli group $\mathcal{P}_{1,\mathtt{p}}$  is both just nonabelian and minimal nonabelian when  $\mathtt{p}\neq 2$. \end{theorem}

\begin{proof} Consider $\mathcal{P}_{1,2}$.  The result follows from Lemmas \ref{l:1} and  \ref{l:3}. Now consider $\mathcal{P}_{1,\mathtt{p}}$ and note that $\mathcal{P}_{1,\mathtt{p}} \simeq \mathbb{H}(\mathbb{Z}(\mathtt{p}))$. Now  
Remarks \ref{nice} and \ref{r:2} (or even the classifications known for these classes of groups in \cite{rob} and \cite{kos}) show that the result is true. \end{proof}

Because of Remark \ref{r:2} and the presentation in Lemma \ref{l:1}, one can see that there are just 6 nonabelian subgroups in $\mathcal{P}_{1,2}$, three of these are isomorphic to $D_8$ and the remaining three are isomorphic to $Q_8$. All the other maximal subgroups of $\mathcal{P}_{1,2}$ are abelian. This means that we have no hope to find that all subgroups of $\mathcal{P}_{1,2}$ are abelian, but the situation is much better when we look at quotients because of Theorem \ref{main1additional}. Thanks to the preliminary Lemmas \ref{l:1},  \ref{l:2} and \ref{l:3}, we may prove another theorem of decomposition and this time for Pauli groups $\mathcal{P}_{n,2}$ with $n \ge 1$ arbitrary.

\begin{theorem}\label{main1} For all $n \ge 1$  there are normal subgroups $H_1, H_2, \ldots, H_n$ in $\mathcal{P}_{n,2}$ such that $$\mathcal{P}_{n,2}= ...(((H_1 \bullet_{L_1} H_2) \bullet_{L_2} H_3) \bullet_{L_3} H_4)\ldots $$ with $L_j\simeq [\mathcal{P}_{j+1, 2}, \mathcal{P}_{j,2}]$ abelian normal subgroup and $H_j  \simeq \mathcal{P}_{1,2}$  for all $j=1,2, \ldots, n$. \end{theorem}

\begin{proof} We proceed by a recursive argument, so we offer a constructive method at the same time. The case $n=1$ is trivial (compare with Lemma \ref{l:1} and Remark \ref{r:2}). Consider $n=2$ and the Pauli group $\mathcal{P}_{2,2}$ with 64 elements. Using the matrix representation of $\mathcal{P}_{1,2}$ and looking at Lemma \ref{l:1}, we get
\[Z(\mathcal{P}_{1,2}) = \left \{\left( \begin{array}{ll}{1} & {0} \\ {0} & {1}\end{array}\right), \left( \begin{array}{ll}{-1} & {0} \\ {0} & {-1}\end{array}\right), \left( \begin{array}{ll}{i} & {0} \\ {0} & {i}\end{array}\right), \left( \begin{array}{ll}{-i} & {0} \\ {0} & {-i}\end{array}\right) \right \} \simeq \Z(4).\]
Having in mind the proof of Lemma \ref{l:1} and denoting by $X, Y, Z$ the Pauli matrices in \eqref{paulimatricesxyz}, we have seen that $Z(\mathcal{P}_{1,2})=\langle XYZ \rangle$. The generators of $\mathcal{P}_{2,2}$ are $\{X \otimes I, Y \otimes I, Z \otimes I, I \otimes Z, I \otimes X\}$ and $\mathcal{P}_{2,2}$ has the following matrix representation in $\mathrm{SL}(4,\mathbb{C})$
 \begin{gather*}
A= X\otimes I= \left( \begin{array}{cccc}{0} & {0} & {1} & {0} \\ {0} & {0} & {0} & {1} \\ {1} & {0} & {0} & {0} \\ {0} & {1} & {0} & {0}\end{array}\right), 
 B= Y \otimes I= \left( \begin{array}{cccc}{0} & {0} & {-i} & {0} \\ {0} & {0} & {0} & {-i} \\ {i} & {0} & {0} & {0} \\ {0} & {i} & {0} & {0}\end{array}\right),\\ 
C= Z \otimes I=  \left( \begin{array}{cccc}{1} & {0} & {0} & {0} \\ {0} & {1} & {0} & {0} \\ {0} & {0} & {-1} & {0} \\ {0} & {0} & {0} &{-1}\end{array}\right),\\ 
D= I\otimes Z = \left( \begin{array}{cccc}{1} & {0} & {0} & {0} \\ {0} & {-1} & {0} & {0} \\ {0} & {0} & {1} & {0} \\ {0} & {0} & {0} & {-1}\end{array}\right), 
E = I \otimes X =  \left( \begin{array}{cccc}{0} & {1} & {0} & {0} \\ {1} & {0} & {0} & {0} \\ {0} & {0} & {0} & {1} \\ {0} & {0} & {1} & {0}\end{array}\right).
\end{gather*}
Let's note a general fact:  the commutator between two matrices,  may be written as \[[A,B]=A^{-1}B^{-1}AB=A^{-1}A^B,\]
where \(A^B=B^{-1}AB\) denotes the conjugate of $A$ by $B$. Note that we have always  the rule ${[A,B]}^{-1}=[B,A]$, when we form the inverse of a commutator. Moreover \[[AB,C]={[A,C]}^B \ [B,C].\]
Details can be found in \cite{rob}. According to \cite{gap}, \(\mathcal{P}_{2,2}\) is the group No. 266 among the nonabelian ones of order \(64\), that is,  \begin{equation}\label{newpres}\mathcal{P}_{2,2}=\langle A,B,C,D,E \ | \  A^2=B^2=C^2=D^2=E^2=I,\end{equation}\[ [ABC,A]=[ABC,B]=[ABC,C]=[ABC,D]=[ABC,E]=I \rangle\] with centre $Z(\mathcal{P}_{2,2}) \simeq Z(\mathcal{P}_{1,2}) \simeq \mathbb{Z}(4)$. The above presentation is realised with the minimum number of generators, and we may apply the aforementioned commutator identities, and
\[I=[ABC,A]={[AB,A]}^C \ [C,A]={({[A,A]}^B \ [B,A])}^C \ [C,A]={[B,A]}^C \ [C,A]\]
in order to get the additional relations
%\[[A,C]={[B,A]}^C=-I.\]
%Similarly one can derive the additional relations 
\begin{equation}\label{newpres2}[A,B]=[D,E]=[B,C]=[C,A]=-I\end{equation}  \[[C,E]=[B,E]=[A,D]=[A,E]=[C,D]=I.\]
Now we show an additional method of construction, thanks to Lemma \ref{l:1}. We get
$$H_1=\langle A,B,C \ | \ A^2=B^2=C^2=I, (BC)^4=(CA)^4=(AB)^4=I \rangle \simeq \mathcal{P}_{1,2};$$
and rewrite this group in the following way
$$H_1=\langle u,x,y \mid u^4 = x^2 = 1, u^2 = y^2, x^{-1}ux = u^{-1}, uy = yu, xy = yx \rangle \simeq \mathcal{P}_{1,2};$$
where $u=AB,   \ x=B,  \  y=ABC;$
now we take another copy $H_2$ of $H_1$ with $u=AB,   \ z=B,  \  t=ABC;$
%$$H_1=H_2=\langle A,D,E \ | \ A^2=D^2=E^2=1, (DE)^4=(EA)^4=(AD)^4=1 \rangle \simeq \mathcal{P}_{1,2};$$
$$H_2=\langle u,z,t \mid u^4 = z^2 = 1, u^2 = t^2, z^{-1}uz = u^{-1}, ut = tu, zt = tz \rangle \simeq \mathcal{P}_{1,2}.$$
%\textcolor{red}{where $u^{-1}={(AB)}^{-1}=AD,   \ z=D,  \  t=ADE.$ }
Note that 
\begin{gather*}
Z(H_1)=\langle ABC \rangle = \left\langle \left( \begin{array}{cccc}{i} & {0} & {0} & {0} \\ {0} & {i} & {0} & {0} \\ {0} & {0} & {i} & {0} \\ {0} & {0} & {0} & {i}\end{array}\right)\right\rangle \simeq Z(H_2) \simeq \Z(4).
\end{gather*}
%but we are not yet ready to conclude that
%$$Z(\mathcal{P}_{2,2})  \simeq \Z(4)  \ \ \ \quad \mathrm{and} \ \ \ \quad [\mathcal{P}_{2,2},\mathcal{P}_{2,2}] \simeq \Z(4).$$
We are going to show  $\mathcal{P}_{2,2} \simeq H_1 \bullet_{L_1} H_2$ with $L_1 \subseteq Z(\mathcal{P}_{2,2})$.
First of all, we define  the set of generators
$ \mathrm{gen}\left(\mathcal{P}_{1,2}\right)=\{u,x,y\}=\{U,X_1,Z_1\} $ with $u=U$, $X_1=x$  and $Z_1=y$, then
\begin{align}
\label{generators}
\mathrm{gen}\left(\mathcal{P}_{n,2}\right) &= \left \{ U, X_1, X_2, \ldots, X_n, Z_1, Z_2, \ldots, Z_n \right \} \\
&=\left\{U, X_{1}, Z_{1}\right\} \cup\left\{U, X_{2}, Z_{2}\right\} \cup \ldots \cup\left\{U, X_{n}, Z_{n}\right\}
\end{align}
and  $\langle U,X_1,Z_1 \rangle \simeq \langle U,X_2,Z_2 \rangle  \simeq \ldots \simeq \langle U,X_n,Z_n \rangle \simeq \mathcal{P}_{1,2}$, or equivalently $\langle U,X_1,Z_1 \rangle \simeq H_1$, $\langle U,X_2,Z_2 \rangle \simeq H_2$ and so on with all $H_j \simeq \mathcal{P}_{1,2}$. 
Then we denote by $R\left(D_{8}\right)$ the relations  in a presentation like \eqref{d8} for $D_8$, with $R(\mathbb{Z}(4))$ the relation on the generator $Z_j$ such that its fourth power equals one and  with $R\left (U,X_1,Z_1 \right) $  the relations for $\mathcal{P}_{1,2}$ in a presentation like  \eqref{g1}. This means that $R\left (U,X_1,Z_1 \right )$ may be written as
\[
R\left (U,X_1,Z_1 \right ) = R\left(D_{8}\right) \cup R(\mathbb{Z}(4)) \cup \{ Z_1^2 = U^2  \} \cup \{[X_1,Z_1]=1\} \cup \{[U,X_1]=1\},
\]
where $Z_1^2 = U^2$ shows that we are identifying the center of the factor $\simeq D_8$ with a cyclic subgroup of order two in the other factor $\simeq \Z(4)$ when we form the weak central product (see Example \ref{r:2}). The remaining relations show that $X_1$ commutes both with $Z_1$ and $U$. More generally, we have the same behaviour for all $j=1,2,\dots,n$
\[
R\left (U,X_j,Z_j \right ) = R\left(D_{8}\right) \cup R(\mathbb{Z}(4)) \cup \{ Z_j^2 = U^2  \} \cup \{[X_j,Z_j]=1\} \cup \{[U,X_j]=1\}.
\]
Therefore 
\begin{equation}
\label{presentation}
\mathcal{P}_{n,2} =  \langle \mathrm{gen}\left(\mathcal{P}_{n,2}\right) \ |  \ R(U,X_j,Z_j) \ \textrm{ for all } j=1,2,\dots,n \rangle
\end{equation}
Note that all the other relations can be derived from \eqref{presentation} in a general presentation for a Pauli group of type $\mathcal{P}_{n,2}$.
In particular, we have  for $n=2$ that $\mathrm{gen}(\mathcal{P}_{2,2})=\{U,X_1,Z_1,X_2,Z_2\}$ so the cardinality of this set is five (compare with \eqref{newpres}). Moreover 
\[[X_1, X_2] = [Z_1, Z_2] = U^2\] 
\[Z\left(\left\langle U, X_{1},Z_1 \right\rangle\right)=\left\langle Z_{1}\right\rangle=Z(H_1) \simeq  Z\left(\left\langle U, X_2,Z_2,\right\rangle\right)=\left\langle Z_{2}\right\rangle=Z(H_2) \simeq \langle U \rangle \simeq \mathbb{Z}(4) \] and so
\[\mathcal{P}_{2,2} \simeq \mathcal{P}_{1,2} \bullet_{\langle U \rangle} \mathcal{P}_{1,2}. \]
%where it is possible to interpret $\bullet_{\langle U\rangle}$ as an appropriate equivalence relation $\sim$ in 
%$\left \langle Z_1  \right \rangle  \times \left \langle Z_2  \right \rangle \simeq \Z(4) \oplus \Z(4)$, that is, 
%$$
%\left \langle Z_1  \right \rangle \bullet_{\langle U\rangle} \left \langle Z_2  \right \rangle=\frac{
%\left \langle Z_1  \right \rangle \times \left \langle Z_2  \right \rangle}{\sim}$$
%$$  = \big \{ (1,1), (1,\langle Z_1\rangle), (1, \langle Z^2_1\rangle), (1,\langle Z^3_1\rangle),
%(\langle Z_1\rangle,1), (\langle Z_1\rangle,\langle Z_2\rangle),(\langle Z_1\rangle,\langle Z^2_2\rangle ), (\langle Z_1\rangle, %\langle Z_2^3\rangle),$$
%$$(\langle Z_1^2\rangle, 1), (\langle Z_1^2\rangle, \langle Z_2\rangle), (\langle Z_1^2\rangle, \langle Z_2^2\rangle ), (\langle Z_1^2\rangle , \langle Z_2^3\rangle ), 
% (\langle Z_1^3\rangle, 1), (\langle Z_1^3\rangle, \langle Z_2\rangle),  (\langle Z_1^3\rangle , \langle Z_2^2\rangle  ), 
% $$
% $$(\langle Z_1^3\rangle, \langle Z_2^3\rangle ) \big \}/\sim$$
% $$\simeq \Z(4) \oplus \Z(2)$$
Note that for $\mathcal{P}_{2,2}$ we have directly found a cyclic group \[\Z(4) \simeq L_1=\langle U\rangle=[H_1,H_2] \subseteq H_1 \cap H_2 \] such that
$$L_1=[H_1, H_2]  = \langle U \rangle  \subseteq Z(\mathcal{P}_{2,2})$$
(actually the equality holds above) suitable for  decomposition of $\mathcal{P}_{2,2}$ in terms of Definition \ref{def:1} with two factors $H_1\simeq \mathcal{P}_{1,2}$  and $H_2 \simeq \mathcal{P}_{1,2}$ identified via their common subgroup $L_1$. The thesis follows for $n=2$.

Because of the constructive method, we can go ahead with $n=3$ and find an $H_3=\langle  U, X_3, Z_3 \ |   \ R(U,X_3,Z_3)\rangle \simeq \mathcal{P}_{1,2}$ such that $\mathcal{P}_{3,2} = (H_1 \bullet_{L_1} H_2) \bullet_{L_2} H_3$ with \[L_2=[(H_1 \bullet_{L_1} H_2), H_3] \simeq [\mathcal{P}_{2,2},\mathcal{P}_{1,2}]\subseteq  (H_1 \bullet_{L_1} H_2) \cap H_3\] which turns out to be an abelian group of order $4$ contained in $Z(\mathcal{P}_{3,2})$. Note that for any two normal subgroups $\Gamma_1, \Gamma_2$ of a given group $\Gamma$, we always have  $[\Gamma_1, \Gamma_2] \subseteq \Gamma_1 \cap \Gamma_2$. In our case, this implies that 
\[|[(H_1 \bullet_{L_1} H_2), H_3] | \le |(H_1 \bullet_{L_1} H_2) \cap H_3|.\]
Again we can check the conditions in Definition \ref{def:1} and find the required decomposition for $\mathcal{P}_{3,2}$.  For larger Pauli groups, the argument is the same. The set of generators for $\mathcal{P}_{n,2}$ is given by $2n+1$ Pauli matrices $2^n \times 2^n$, described by the set  in \eqref{generators}
and such that
$$H_1=\langle U,X_1,Z_1 \ |   \ R(U,X_1,Z_1) \rangle \simeq \mathcal{P}_{1,2}; \ \ H_2=\langle U,X_2, Z_2 \ |   \ R(U,X_2,Z_2)  \rangle \simeq \mathcal{P}_{1,2}; \ \ \ldots
$$
$$
H_n=\langle U,X_n, Z_n \ |   \ R(U,X_n,Z_n) \rangle \simeq \mathcal{P}_{1,2}$$
and then we put all the relations together, obtaining \eqref{presentation}. The result follows.
\end{proof}

 We can modify the argument of Theorem \ref{main1} for  $\mathcal{P}_{n,\mathtt{p}}$ when $\mathtt{p}$ is an odd prime.

\begin{corollary}\label{nicebis} 
%\RED{Given an odd prime $\mathtt{p}$, $m=1$ and $n \ge 1$, the Pauli group $\mathcal{P}_{n,\mathtt{p}}$ is  isomorphic to a Heisenberg group $\mathbb{H}({\mathbb{Z}(\mathtt{p})}^n)$ of order $\mathtt{p}^{2n+1}$. }
Given an odd prime $\mathtt{p}$ and $n, m \ge 1$, the Pauli group $\mathcal{P}_{n,\mathtt{p}^m}$ is  isomorphic to a Heisenberg group $\mathbb{H}({\mathbb{Z}(\mathtt{p}^m)}^{n})$ of order $\mathtt{p}^{{2nm}+1}$. 
\end{corollary}

\begin{proof}
We can use an argument as in Theorem \ref{main1}. First we consider $\mathcal{P}_{1,\mathtt{p}}$ and we  observe that 
$\mathcal{P}_{1,\mathtt{p}} \simeq \mathbb{H}(\mathbb{Z}(\mathtt{p}))$.
Then we consider  $\mathcal{P}_{2,\mathtt{p}}$ and find that $$\mathcal{P}_{2,\mathtt{p}}=\mathcal{P}_{1,\mathtt{p}} \bullet \mathcal{P}_{1,\mathtt{p}} \simeq \mathbb{H}(\mathbb{Z}(\mathtt{p})) \bullet \mathbb{H}(\mathbb{Z}(\mathtt{p})) \simeq \mathbb{H}({\mathbb{Z}(\mathtt{p})}^2) $$
This is a $\mathtt{p}$-group with center of order $\mathtt{p}$, exponent $\mathtt{p}$ and order $\mathtt{p}^5$. Iterating this observation and noting that
$$\mathbb{H}({\mathbb{Z}(\mathtt{p})}^2) \bullet \mathbb{H}(\mathbb{Z}(\mathtt{p})) 
\simeq \mathbb{H}({\mathbb{Z}(\mathtt{p})}^3), $$
 the result for $m=1$ follows.

The general case of $\mathcal{P}_{n,\mathtt{p}^m}$  follows by induction.
\end{proof}

One of the consequences of the above results is expressed below.
 
\begin{corollary}\label{jna} The Pauli groups $\mathcal{P}_{n,2}$ are just nonabelian if and only if $n=1$. On the other hand,   $\mathcal{P}_{n,\mathtt{p}}$ with  $\mathtt{p}$ odd is just nonabelian for all $n \geq 1$.
\end{corollary}
\begin{proof}
The sufficient condition is clear from Theorem \ref{main1additional}. The necessary condition can be extrapolated from the argument in the proof of Theorem \ref{main1}. In fact any time $n \ge2$, $\mathcal{P}_{n,2}$ possesses always a nontrivial normal subgroup $N$ which is isomorphic to $\mathcal{P}_{1,2}$ and $G/N \simeq G_{n-1}$ is manifestly nonabelian. In case of odd primes, it is sufficient to apply Corollary \ref{nicebis} and Lemma \ref{l:3}.
\end{proof}

\section{Applications of the main results}
\label{sec:applications}

In this section we give two applications of our main results. First we show that our decomposition can be applied to the `lifted' Pauli groups of Gottesman and Kuperberg~\cite{gottesmankuperberg, gottesmantalk}. Second, we show how to detect families of abelian subgroups directly from the decomposition in Theorem \ref{main1} and Corollary \ref{nicebis}.

\subsection{Decomposition result for `lifted' Pauli groups}

Stabiliser codes for quantum error correction were introduced in \cite{gottesman1997stabilizer} and constitute an active field of research in quantum information theory. 
A stabiliser is an abelian subgroup of the Pauli group and the error correcting properties of stabiliser codes are firmly rooted in the group structure of Pauli groups.
The reader may refer to \cite{gottesman2010introduction} for a formal introduction to quantum error correction and stabiliser codes.

Quantum error-correcting codes on the qudits enable a greater variety of codes~\cite{chau, gottesman1999, knill, rains} and it is interesting to consider stabiliser codes for qudits.
% For a system of $n$-dimensional qudits we define a stabiliser $\mathcal{S}$ to be an Abelian subgroup of $G_{n,d}$ that does not contain $\omega 1$ (recall that $\omega = \exp(2 \pi i /d)$).
%The corresponding stabiliser code $C$ is the set of quantum states $\psi$ 
%(also known as \textit{codewords}) such that
%\[C = \left \{ \psi \mid S \psi = \psi \textrm{ for all }S \in \mathcal{S}   \right \} \]
%if $|C|=1$ then the $C$ is a stabiliser state.
It is possible to extend the notion of stabiliser codes to qudits of any dimension~\cite{knill} but part of the structure of the qubit stabiliser is lost.
An interesting case where it is possible to extended the theory of qubit stabilisers without loss---by exploiting the fact that there is a unique finite field $\mathbb{F}_\mathtt{q}$ for every prime power $\mathtt{q}=\mathtt{p}^m$---is for $\mathtt{q}$-dimensional qudits where $\mathtt{q}$ is a prime power~\cite{ashi,ketkar}.
Gottesman and Kuperberg recently discussed how the standard way of extending stabiliser codes for prime power qudits does not exploit the full field structure and propose a new definition of stabiliser codes, based on the notion of `lifted' Pauli groups, over $\mathtt{q}$-dimensional registers~\cite{gottesmankuperberg, gottesmantalk}.

We present the definition of `lifted' Pauli groups given in~\cite{gottesmankuperberg,gottesmantalk}.

\begin{definition}
 Let $\mathtt{q}=\mathtt{p}^m$ where $p$ is an odd prime, and let $n\geq 1$. The \textit{lifted Pauli group} is the group of unitriangular matrices 
 \[
 \hat{\mathcal{P}}_{n,\mathtt{q}}=\left\{\left(\begin{array}{ccc}{1} & {\vec{\alpha}} & {\eta} \\ {0} & {I} & {\vec{\beta}^{T}} \\ {0} & {0} & {1}\end{array}\right)\right\}
 \]
 where $\vec{\alpha}, \, \vec{\beta} \in \mathbb{F}_\mathtt{q}^n$, $[\cdot]^T$ indicates the transpose of a vector, and $\eta \in \mathbb{F}_\mathtt{q}$.
\end{definition}
Elements of the lifted Pauli group are also denoted by 
\[
P = \left(\begin{array}{ccc}{1} & {\vec{\alpha}} & {\eta} \\ {0} & {I} & {\vec{\beta}^{T}} \\ {0} & {0} & {1}\end{array}\right) = \omega^\eta X ^ {\vec{\alpha}} Z^ {\vec{\beta}}
\]
for an element $P\in\hat{\mathcal{P}}_{n, \mathtt{q}}$. Note that in the last expression $X$ and $Z$ are just symbols and not Pauli operators; similarly, $\omega$ is just a symbol and not a root of unity.
Using a multiplication rule similar to the one used for the standard Pauli group that is is possible to show that there exist a surjective homomorphism 
\begin{equation}\label{definitionofpi}
\Pi : \omega^{\eta} X^{\vec{\alpha}} Z^{\vec{\beta}} \in \hat{\mathcal{P}}_{n, \mathtt{q}} \longmapsto \Pi\left(\omega^{\eta} X^{\vec{\alpha}} Z^{\vec{\beta}}\right)=\omega^{\operatorname{tr} \eta} X^{\vec{\alpha}} Z^{\vec{\beta}}
\in \mathcal{P}_{n, \mathtt{q}}.
\end{equation}

Therefore we have the following relevant result.

\begin{corollary}\label{nicetris}
Assume $\widehat{\mathcal{P}}_{n,\mathtt{q}}$ is a lifted Pauli group of $\mathcal{P}_{n,\mathtt{q}}$, where $\mathtt{q}=\mathtt{p}^m$ and $\mathtt{p}$ odd prime. Then $\widehat{\mathcal{P}}_{n,\mathtt{q}}$ is isomorphic to $\mathbb{H}({\mathbb{Z}(\mathtt{p^m})}^n)$ up to a quotient.
\end{corollary}
\begin{proof}
This  follows from an application of Corollary \ref{nicebis} and from the First Isomorphism Theorem for groups, applied to the epimorphism \eqref{definitionofpi}.
\end{proof}

In analogy to \eqref{definitionofpi}, one can define an epimorphism as in \cite{gottesmankuperberg} when we deal with $\mathtt{p}=2$, 
\[
\varepsilon :  \hat{\mathcal{P}}_{n, 2} \mapsto  \mathcal{P}_{n, 2}.
\]
and we have  by Theorem \ref{main1} that:

\begin{corollary}\label{nicetetra}
The lifted Pauli group $\widehat{\mathcal{P}}_{n,2}$ is isomorphic to the weak central product of $n$ copies of  $\mathcal{P}_{1,2}$ up to a quotient.
\end{corollary}

In particular, we can see that recognising abelian subgroups in lifted Pauli groups is reduced to recognise those in the usual Pauli groups up to a quotient. In the successive subsection we will see how to do this via Theorems \ref{main1additional} and \ref{main1}.

\subsection{Identification of abelian subgroups}

In this section we show how our main results can be used to detect families of abelian subgroups. This may have applications in the theory of mutually unbiased basis~\cite{durt}.

\begin{corollary}\label{mub}
The Pauli group $\mathcal{P}_{n,\mathtt{p}^m}$  is minimal nonabelian for all $\mathtt{p} \neq 2$ and $m \ge 1$. Moreover, it contains always two distinct maximal abelian normal subgroups $A,B$ and a nonnormal abelian subgroup $H$ such that $$\mathcal{P}_{n,\mathtt{p}^m}=A \rtimes H=B \rtimes H,$$ 
where $[A,B]=A \cap B = Z (\mathcal{P}_{n,\mathtt{p}^m})$ and $A \simeq B$.
\end{corollary}

\begin{proof} Apply Corollary \ref{nicebis}. Now the considerations in the previous Paragraph 2.2 along with Remark \ref{nice} conclude the proof.
\end{proof}

Note that $\mathcal{P}_{n,\mathtt{p}^m}$ are also minimal nonabelian when $\mathtt{p}\neq 2$, since so are Heisenberg groups. In fact they satisfy the classification results in \cite{rob} being extraspecial $\mathtt{p}$-groups.

In order to study the case $\mathtt{p}=2$ we must recall some facts on dihedral groups from \cite{rob}. We first introduced the dihedral group  $D_8 $ in Lemma~\ref{l:1}. Fig.~2 contains the lattice of subgroups of $D_8=\langle u,a \rangle$ which is involved in the central product with $\mathbb{Z}(4)$ for the formation of $\mathcal{P}_{1,2}$. The lattice of  subgroups of $D_8$ is
$$\mathcal{L}(D_8)=\{1, \langle u \rangle, \langle u^2\rangle, \langle a \rangle, \langle ua \rangle, \langle u^2a \rangle, \langle u^3a \rangle, \langle u^2,a\rangle, \langle u^2,ua \rangle, D_8 \}.$$
The normal subgroups in $D_8$ are of course $D_8$ and $1$, but also $U=\langle u \rangle$, $Z(D_8)=\langle u^2 \rangle$, $M_1=\langle u^2,a\rangle$ and $M_2=\langle u^2,ua \rangle$.  Then there are subgroups of order two below these, namely $V,W,H,K$.
It is clear the analogy with the left side of Fig.~2.

\begin{center}
\begin{tikzpicture}
%[scale=.8,auto=left,every node/.style={circle,fill=blue!10}]
[scale=.8,auto=left]
  \node (a) at (0,0) {$D_8$};
  \node (b) at (-2,-2) {$M_1$};
  \node (c) at (0,-2) {$U$};
  \node (d) at (2,-2) {$M_2$};
  \node (e) at (-4,-4) {$H$};
  \node (f) at (-2,-4) {$K$};
  \node (g) at (0,-4) {$Z(D_8)$};
  \node (h) at (2,-4) {$V$};
  \node (i) at (4,-4) {$W$};
  \node (j) at (0,-6) {$1$};

%  \node (Fig) at (-10,2) {$\Gamma(S_3)$};
  
  %\draw (a) -- (b) -- (c)--(a);
  \draw (a) -- (b);
  \draw (a) -- (c);
  \draw (a) -- (d);
  \draw (b) -- (e);
  \draw (b) -- (f);
  \draw (b) -- (g);
  \draw (c) -- (g);
  \draw (d) -- (g);
  \draw (d) -- (h);
  \draw (d) -- (i);
  \draw (e) -- (j);
  \draw (f) -- (j);
  \draw (g) -- (j);
  \draw (h) -- (j);
  \draw (i) -- (j);
 \end{tikzpicture}
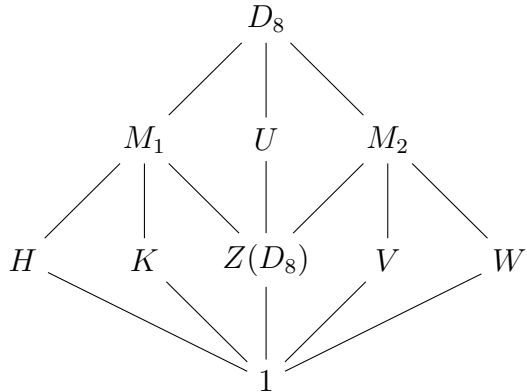
\centerline{\textbf{Figure 3}: Hasse diagram of $D_8$. }
 \end{center}
%\end{center}
\vspace{0.5em}

Denoting by $\sigma(r)$ the sum of divisors of $r$ and by $\tau(r)$ the number of divisors of $r$, it is known that the number of subgroups of $D_8$ can be counted as
  \begin{equation}\label{counting1}|\mathcal{L}(D_8)|=\sigma(4) + \tau(4)= 7 +3 = 10
\end{equation}
and all of them are abelian up to $D_8$, so the number of nontrivial abelian subgroups of $D_8$ is $$c_{\mathrm{ab}}(D_8)=8.$$

\begin{corollary}\label{mubbis}
The number of nontrivial abelian subgroups $c_{\mathrm{ab}}(\mathcal{P}_{n,2})$  of $\mathcal{P}_{n,2}$  is bounded by
  $$ 2(c_{\mathrm{ab}}(\mathcal{P}_{n-1,2})+1) \ge c_{\mathrm{ab}}(\mathcal{P}_{n,2}) \ge 10n.$$
\end{corollary}

\begin{proof} We begin by proving the lower bound. Apply Theorem \ref{main1}, noting that each factor $H_j$ in the decomposition in weak central product contains at least one dihedral group of order $8$. We get $8n$ abelian subgroups in this way. Now consider Fig.~2 and note that in each $H_j$ we also find two  abelian subgroups like $\langle b \rangle$ and $\langle ub\rangle$ of order two,  so  the result follows.

For the upper bound, we refer again to Fig.~2 and  Lemma \ref{l:1}. We observe from  \cite{gap} that all nontrivial abelian subgroups of  $\mathcal{P}_{1,2}$ are the following:   the unique normal subgroup $\langle u^2 \rangle$ of order two;  the six nonnormal subgroups of order two $\langle a \rangle$, $\langle ua \rangle$, $ \langle u^2a \rangle,$ $ \langle u^3a \rangle$, $\langle ub \rangle$ and $\langle u^3b \rangle$; four cyclic subgroups of order four $\langle b \rangle$, $\langle u \rangle$, $\langle ab \rangle$ and $\langle uab \rangle$; three $2$-elementary abelian $2$-subgroups of rank two $\langle u^2, a \rangle$, $\langle u^2, ua \rangle$ and $\langle u^2, ub \rangle$; finally
    the subgroups of order eight $\langle u,b \rangle$, $\langle a, b \rangle$ and $\langle ua, b \rangle$. This shows that
    $$c_{\mathrm{ab}}(\mathcal{P}_{1,2})=1+6+4+3+3=17 \le 2 \  c_{\mathrm{ab}}(D_8) + 2,$$
    because in the process of forming the central product we find at least two   subgroups in $\mathcal{P}_{1,2}$ isomorphic to $D_8$. 
    By Theorem \ref{main1} we have that
$$ c_{\mathrm{ab}}(\mathcal{P}_{2,2}) \le 2 \  c_{\mathrm{ab}}(\mathcal{P}_{1,2})+2$$
 because in the process of forming the central product we find at least two   subgroups in $\mathcal{P}_{1,2}$  and a subgroup of order $4$ which gives the additional term equal to $2$.  Therefore we find that
$$ c_{\mathrm{ab}}(\mathcal{P}_{n,2}) \le 2 \  c_{\mathrm{ab}}(\mathcal{P}_{n-1,2})+2$$
from which the result follows.
\end{proof}

Note that there exist more sophisticated counting formulas than the one we presented (e.g. ~\cite{berk}). Our purpose was to show a simple application of our main results for the problem of counting abelian subgroups of Pauli groups. This follows from the dihedral component of the groups in the factorisation given in Theorem~\ref{main1}.

\proof[\textbf{Acknowledgements}]
We acknowledge the referee and Drew S. Vandeth for relevant observations. This research was supported in part by the National Science Foundation under Grant No. NSF PHY-1748958 and by the Heising-Simons Foundation. A.R. is supported by the Simons Foundation through \textit{It from Qubit: Simons Collaboration on Quantum Fields, Gravity, and Information}.

\vspace{1em}


\begin{thebibliography}{20}
%\bibitem{BMPS} P. Bianucci, C. Miquel, J. P. Paz, and M. Saraceno, Discrete
%Wigner functiosn and the phase space representation of quantum
%computers. quant-ph/0106091.

%\bibitem{alireza1}A. Abdollahi, S. Akbari and H.R. Maimani,Non-commuting graph of a group, \textit{J. Algebra} \textbf{298} (2006), 468--492.
\bibitem{aaronson} S. Aaronson and D. Gottesman, Improved simulation of stabilizer circuits, \textit{Phys. Rev. A},\textbf{70(5)} (2004).

\bibitem{ashi} A. Ashikhmin and E. Knill, Nonbinary quantum stabilizer codes, \textit{IEEE Trans.  Inf. Theory}, \textbf{47}(2001).

\bibitem{bagrus1} F. Bagarello and F.G. Russo, A description of pseudo-bosons in terms of nilpotent Lie algebras, \textit{J. Geom. Phy.} \textbf{125} (2018).


\bibitem{berk}Y. Berkovich, \textit{Groups of Prime Power Order}, Vol. I, de Gruyter, 2018. 

%\bibitem{BKRS} G. Bj\"ork, A. B. Klimov, J. L. Romero, and L. L. Sanchez-Soto, On the structure of the sets of mutually unbiased bases for $N$ qubits. quant-ph/0508129v1.


%\bibitem{BLZ} C. Brukner, J. Lawrence and A. Zeilinger, Mutually unbiased binary observables sets on $N$ qubits, \textit{Phys. Rev. A} \textbf{65} (2002), 032320.


%\bibitem{BZ} C. Brukner, and A. Zeilinger, Mutually unbiased binary observables sets on $N$ qubits, quant-ph/0104012v2.

%\bibitem{CMS} S. Chaturvedi, N. Mukunda and R. Simon, Wigner distributions for non Abelian finite groups of odd order, \textit{Phys. Lett. A} \textbf{ 321} (2004).

\bibitem{chau} H.F. Chau, Correcting quantum errors in higher spin systems, \textit{Phys. Rev. A}, \textbf{55} (1997).

%\bibitem{CN} I.L. Chuang and M. A. Nielsen, \textit{Quantum computation and quantum information}, Cambridge University Press, Cambridge, 2000.

%\bibitem{DD} J. Dehaene and B. De Moor, The Clifford group, stabilizer states, and linear and quadratic operations over GF(2), \textit{Phys. Rev. A} \textbf{68} (2003).

%\bibitem{D1} T. Durt, A new solution for the Mean King's problem,quant-ph/0401037v2

\bibitem{durt} T. Durt, B.G. Englert, I. Bengtsson, and K. Å»yczkowski, On mutually unbiased bases, \textit{Int. J. Quantum Inf.}, \textbf{8(04)} (2010).

\bibitem{GHW}  K. S. Gibbons, M. J. Hoffman and W. K. Wootters, Discrete phase space based on finite fields, \textit{Phys. Rev. A} \textbf{70} (2004).

\bibitem{gottesman2010introduction}  D.Gottesman, An introduction to quantum error correction and fault-tolerant quantum computation, \textit{Proceedings of Symposia in Applied Mathematics} \textbf{68} (2010).


\bibitem{gottesman1997stabilizer}D. Gottesman, Stabilizer codes and quantum error correction, Caltech PhD Thesis, preprint, arXiv: quant-ph/9705052, 1997.

\bibitem{gottesman1999} D. Gottesman, Fault-Tolerant Quantum Computation with Higher-Dimensional Systems, \textit{Chaos Solitons and Fractals} \textbf{10} (1999).



%\bibitem{gottesmanslides}D. Gottesman, Stabilizer codes for prime power qudits, Available online at: http://www.qec14.ethz.ch/slides/DanielGottesman.pdf


\bibitem{gottesmankuperberg} D. Gottesman and G. Kuperberg, personal communications.

\bibitem{gottesmantalk} D. Gottesman, \textit{Stabilizer codes with prime power qudits}, \textit{Invited Talk at Caltech IQIM Seminar (Pasadena, California)} (2014).


%\bibitem{gross} D. Gross, Finite Phase Space Methods in Quantum Information, MSc Thesis, Universit\"at Potsdam, Germany, 2005.

\bibitem{grundh}T. Grundh\"ofer and M. Stroppel, Automorphisms of Verardi Groups: Small Upper Triangular Matrices over Rings, \textit{Beitr\"age  Algebra  Geom.} \textbf{49} (2008).

\bibitem{haah} J. Haah, Algebraic methods for quantum codes on lattices, \textit{Revista Colomb. Mat.} \textbf{50}(2016)

\bibitem{hofmor} K. H. Hofmann and S. Morris, \emph{The Structure of Compact Groups}, de Gruyter, Berlin, 2006.

\bibitem{ketkar} A. Ketkar, A. Klappenecker,S. Kumar, and P.K. Sarvepalli, Nonbinary stabilizer codes over finite fields, \textit{IEEE Trans. Inf. Theory} \textbf{52} (2006).


\bibitem{knill} E. Knill, Non-binary unitary error bases and quantum codes, preprint, arXiv:9608048.

\bibitem{kos} L. Kurdachenko, J. Ot\'al and I. Subbotin, \textit{Groups with prescribed quotient groups and associated module theory}, World Scientific, Singapore, 2002.

\bibitem{rains}  E.M. Rains, Nonbinary quantum codes, \textit{IEEE Trans. Inf. Theory} \textbf{45} (1999).

\bibitem{rob}D. Robinson, \textit{A course in the theory of groups}, Springer, Berlin, 1980.

\bibitem{jp} P. Jorrand and M. Planat, On group theory for quantum gates and quantum coherence, arXiv:0803.1911v2


\bibitem{russo}F.G. Russo, On compact just non-Lie groups, \textit{J. Lie Theory} 17 (2007).


\bibitem{gap}  The GAP Group, \textit{GAP---Groups, Algorithms andProgramming}, version 4.4, available at http://ww.gap-system.org,2005.


\bibitem{Woo1} W. K. Wootters, A Wigner-Function Formulation of Finite-State Quantum-Mechanics, \textit{Ann. Phys.} \textbf{167} (1987).

\bibitem{Woo2} W. K. Wootters, Quantum measurements and finite geometry, \textit{Found. Phys.} \textbf{36} (2006).


\bibitem{z} G. Zauner, Quantendesigns: Grundz\"uge einer nichtkommutativen Designtheorie, Dissertation, Universit\"at Wien, 1999.

\end{thebibliography}
\end{document}